\documentclass[12pt]{article}
\usepackage{amsmath,amsfonts}
\usepackage{hyperref}
\usepackage{amsthm}
\usepackage{enumitem}
\usepackage{amssymb,amsmath,makeidx,verbatim}
\numberwithin{equation}{section}
\newtheorem{theorem}{Theorem}[section]
\newtheorem{corollary}[theorem]{Corollary}
\newtheorem{lemma}[theorem]{Lemma}
\newtheorem{Remark}[theorem]{Remark}
\usepackage[british]{babel}
\date{}
\begin{document}
\title{Gradient estimates for a nonlinear parabolic equation with potential under geometric flow}
\author{Abimbola Abolarinwa\thanks{Department of Mathematics,  University of Sussex, Brighton, BN1 9QH, United Kingdom.}  \thanks{E-mail: a.abolarinwa@sussex.ac.uk}}
\maketitle
\begin{abstract} 
Let $(M, g)$ be an dimensional complete Riemannian manifold. In this paper we prove local Li-Yau type gradient estimates for all positive 
solutions to the following nonlinear parabolic equation 
\begin{equation*}
( \partial_t -  \Delta_g + \mathcal{R}) u( x, t) = - a u( x, t)  \log u( x, t)  
 \end{equation*} 
along the generalised geometric flow. Here $ \mathcal{R} =  \mathcal{R} (x, t)$ is a smooth potential function and $a$ is a constant. As an application we derived a global estimate and a space-time Harnack inequality.
\end{abstract} 
\section{Preliminaries and main results} 
Gradient and Harnack estimates are fundamental tools to tackle classical and modern problems in geometric analysis. These methods applied to parabolic equations were first studied by Li and Yau in their celebrated paper \cite{[LY86]}. They have been applied successfully to the setting of various geometric flows,  for more details see  \cite{[Ab1],[Ab2],[Ba13],[BCP],[CaH09],[CZ11],[HHL12],[KZh08],[Ni07]} and the references therein.  See the paper \cite{[HM10]} for similar applications. In this paper we will drive various gradient estimates for the following nonlinear parabolic equation with potential
\begin{equation}\label{eq52}
\Big(\frac{\partial }{ \partial t} -  \Delta + \mathcal{R}\Big) u( x, t) = - a u( x, t)  \log u( x, t),  
 \end{equation} 
where the symbol $\Delta = \Delta_g $ is the Laplace-Beltrami operator acting on functions in space with respect to metric $g(t)$ in time, $a$ is a constant and $\mathcal{R}: M \times  [0, T] \to \mathbf{R}$ is a $C^\infty$-function on $M$. For instance if we take $\mathcal{R}$ to be the scalar curvature of the manifold and we allow $g$ to evolve by the Ricci flow, $ \partial_t g = - 2Rc$, where $Rc$ is the Ricci curvature tensor, it then reduces to the study of gradient Ricci soliton. Taking $ f = \log u$ a standard calculation yields
\begin{align}\label{eqn2}
\Big(\frac{\partial }{ \partial t}  -  \Delta \Big) f = |\nabla f|^2 - af -  \mathcal{R}.
\end{align}
The study of gradient estimates on $M$ can be reduced  to the study of the properties of the solution $f$ of  (\ref{eqn2}) and it is related to gradient soliton equation \cite{[CaH09],[CZ11]}. Let $(M, g(t)), 0 \leq t \leq T,$ be an $n$-dimensional complete manifold on whose metric $g(t)$ evolves by the geometric flow
\begin{equation}\label{eq51}
 \frac{\partial }{\partial t} g_{ij}(x,t) = 2 h_{ij}(x, t), \hspace{1cm} (x, t) \in M \times [0, T],
 \end{equation}
 where $h_{ij}$ is a general time-dependent symmetric $(0, 2)$-tensor and $T >0$ is taken to be the maximum time of existence for the flow.  In \cite{[Ab1]} we obtain local gradient estimates for 
 \begin{equation}
\Big(\frac{\partial }{ \partial t} -  \Delta_g + \mathcal{R}\Big) u( x, t) = 0
 \end{equation} 
coupled to (\ref{eq51}). In this paper we extend the results to the case of  (\ref{eq52}) under the assumption that the geometry of the manifold remains uniformly bounded throughout the evolution. In particular, our results here can be generalised to  Ricci flow and some other geometric flows on complete manifolds. Indeed, Ricci flow  is a nice setting because of contracted second Bianchi identity that makes the divergence of Ricci tensor to be equal to the half gradient of scalar tensor.

We will impose boundedness condition on the Ricci curvature of the metric $g(t)$. We notice  that when the metric evolves by the Ricci flow, boundedness and sign assumptions are preserved  as long as the flow exists, so it follows that the metrics are uniformly equivalent. Precisely, if $ -K_1 g \leq Rc \leq K_2g$, where $g(t), t \in [0, T]$ is a Ricci flow, then 
\begin{equation}\label{eq2.3}
e^{-k_1 T}g(0) \leq g(t) \leq e^{k_2 T}g(0).
\end{equation}
To see the above bounds (\ref{eq2.3}) we consider the evolution of a vector form $|X|_g = g(X, X), X \in T_xM$. By the equation of the Ricci flow $ \partial_t g(X, X) = - 2Rc(X, X), \ 0 \leq t_1 \leq t_2 \leq T$ and by the boundedness of the Ricci curvature we have 
$ | \partial_t g(X, X) | \leq K_2 g(X, X),$
which implies (by integrating from $t_1\ to\ t_2$)
$$ \Big| \log \frac{g(t_2)(X, X)}{g(t_1)(X, X)}\Big| \leq K_2 t \Big|_{t_1}^{t_2}.$$
Taking the exponential of this estimate with $t_1 =0$ and $t_2 =T$ yields $|g(t)| \leq e^{k_2 T}g(0)$ from which the uniform boundedness of the metric follows. See \cite{[CCG$^+$P2]} and \cite{[CLN06]} for details on the theory of the Ricci flow. Similarly, if there holds boundedness assumption $ -c g \leq h \leq Cg$, the metric $g(t)$ are uniformly bounded below and above for all time $0 \leq t \leq T$ under the geometric flow (\ref{eq51}). Then, it does not matter what metric we use in the argument that follows.
   
We now state the general local space-time gradient estimate  corresponding to those of \cite[Theorem 3.2]{[Ab1]}
\begin{theorem}\label{thm54}(Local gradient estimates).
Let $( M, g(t)), t \in $ be a complete solution to the geometric flow (\ref{eq51}) in some time interval $[0, T].$ Suppose there exist some nonnegative constants $k_1, k_2,$ and $ k_3, $  such that $ R_{ij}(g)  \geq - k_1 g $ and $- k_2 g \leq h \leq k_3 g$   for all $ t \in [0,T]$.  Let $ u \in C^{2,1 } ( M \times [0,T])$  be any smooth positive solution  to  (\ref{eq52}) in the geodesic ball $  \mathcal{B}_{ 2\rho, T}$. Suppose  $ \| \nabla h \|, | \mathcal{R} |, |\nabla \mathcal{R}|  $ and $  | \Delta \mathcal{R} |$ are uniformly bounded on $M \times [0, T]$. Then, the  following estimate holds 
 \begin{align}\label{eq59}
 \left. \begin{array}{l}
\displaystyle \sup_{ x \in \mathcal{B}_{2\rho}} \Big \{ | \nabla f |^2 - \alpha f_t  -\alpha a f - \alpha \mathcal{R} \Big \} \\
\displaystyle \hspace{1cm} \leq \frac{\alpha n  p}{2t}  +    \frac{ \alpha n p }{ 4(\alpha -1) } C_8   +\frac{\alpha n }{2} ( k_2 + k_3)  \varphi \sqrt{pq}\\
\displaystyle \hspace{1cm}+ \frac{\alpha n  p}{2 } \Big \{ \frac{C_9}{\rho^2}  \Big( \frac{ \alpha p}{\alpha - 1}  +  \rho \sqrt{k_1}  + \rho^2 ( k_2 + k_3 )^2 \Big) - a  \Big \}
\end{array} \right.  
\end{align}
for  all $ (x, t) \in  \mathcal{B}_{2\rho, T }, \  t > 0$ and some constants $C_8$ and $C_9$ depending only on $n, \alpha$ and uniform bounds for  $ \| \nabla h \|, |\nabla \mathcal{R}|  $ and $  | \Delta \mathcal{R} |$,  where $ f = \log u$ and $\alpha >  1$ are given such that $ \frac{1}{p} + \frac{1}{q} = \frac{1}{\alpha}$ for any real numbers $p, q > 0.$
\end{theorem}
As an application of the above result we obtain global gradient estimates (Cf. Remark \ref{rmk31}, equation (\ref{eq512})). We then apply the global estimates obtained to derive classical Harnack inequalities by integrating along a space-time path joining any two points in $M$. 
 
The rest of the paper  is as follows; in the next section we state and prove an important lemma that we will apply to prove the theorem above. The lsat section is devoted to the descriptions of the cut-off function needed in the proof,  detail of the proof of Theorem \ref{thm54} itself and its application to Harnack inequality (Cf. Corollary \ref{corN5}).

\section{Important Lemma}
We first proof  the following technical lemma which is a generalization of  Lemma 3.1 in \cite{[Ab1]}. It is originally proved for heat equation on static metric by Li and Yau \cite{[LY86]}. This is very crucial to derivation of both local and global estimate of Li-Yau type. 
\begin{lemma}\label{lem53}
Let $( M, g(t))$ be a complete solution to the generalized flow (\ref{eq51}) in some time interval $[0, T].$ Suppose there exist some nonnegative constants $k_1, k_2, k_3, $ and $k_4$ such that $ R_{ij}(g)  \geq - k_1 g $, $- k_2 g \leq h \leq k_3 g$  and $| \nabla h | \leq k_4$  for all $ t \in [0,T]$. For any smooth positive solution $ u \in C^{2,1 } ( M \times [0,T])$ to  equation  (\ref{eq52}) in the geodesic ball $  \mathcal{B}_{ 2\rho, T}$, it holds that 
\begin{equation}\label{eq55}
\left. \begin{array}{l}
\displaystyle ( \Delta  - \partial_t) F  \geq - 2 \langle \nabla f, \nabla F  \rangle -   \frac{2 \alpha t }{ np} (\Delta f )^2  - \frac{F}{t}  - 3 \alpha n^{\frac{1}{2}} k_4 t |\nabla f|  \\
 \displaystyle  \hspace{2cm} - \Big( ( \alpha - 1)(2 k_3 + a) + 2k_1 \Big) t| \nabla f |^2  - \alpha t \Delta \mathcal{R}  \\
\displaystyle \hspace{2cm} - 2 t ( \alpha -1) \langle \nabla f , \nabla R \rangle   - \frac{\alpha n q }{2}  t ( k_2 + k_3)^2   + aF,
\end{array} \right.
\end{equation}
 where $ f = \log u, F  = t  (| \nabla f |^2 - \alpha \partial_t f  - \alpha \mathcal{R} - \alpha a f )$ and $\alpha \geq 1$ are given such that $ \frac{1}{p} + \frac{1}{q} = \frac{1}{\alpha}$ for any real numbers $p, q > 0.$
\end{lemma}
\begin{proof}
Recall from \cite[Lemma 2.1]{[Ab1]} the following evolutions under the flow
\begin{align}\label{eqn34} 
  \displaystyle ( | \nabla f|^2)_t = - 2 h_{ij} f_i f_j + 2 f_i f_{ti} 
  \end{align}
  \begin{align} \label{eqn35} 
   \displaystyle &   ( \Delta f )_t = \Delta (f_t)  - 2 h_{ij}f_{ij} - 2 \langle \ div \ h , \nabla f \rangle +   \langle \nabla H  , \nabla f \rangle,
\end{align}
where $div$ is the divergence operator i.e., $(div\ h)_k = g^{ij} \nabla_i h_{jk}$.  Notice also that 
$ f_t = \Delta f + | \nabla f |^2 - \mathcal{R} - a f $. Taking covariant derivative of $F$ we have 
$$ F_i = t ( 2 f_j f_{ji} - \alpha f_{ti} - \alpha \mathcal{R}_i  - \alpha a f_i) $$
 and with  Bochner-Weitzenb\"ock's formula  
 \begin{equation}
\Delta | \nabla f|^2 =  2  |f_{ij}|^2 +  2  f_j f_{jji} + 2  R_{ij} f_i f_j
\end{equation}
we have
$$ \Delta F = \sum_{i =1}^n F_{ii} = t \Big( 2 f_{ij}^2 + 2 f_j f_{jji} + 2 R_{ij} f_{ij} -  \alpha \Delta (f_t)  - \alpha  \Delta \mathcal{R} - \alpha a \Delta f \Big).$$ 
Using  (\ref{eqn35})  we get 
\begin{align*} 
 \Delta F = &  t \Big[ 2 f_{ij}^2 + 2 f_j f_{jji} + 2 R_{ij} f_i f_j - \alpha( \Delta f)_t - 2 \alpha h_{ij} f_{ij}  \\
 & \hspace{1.5cm} - 2 \alpha (div\ h)_i f_j + \alpha H_i f_j - \alpha \Delta \mathcal{R} - \alpha a \Delta f \Big]  \\
  & = t \Big( 2 f_{ij}^2  - 2 \alpha h_{ij} f_{ij} \Big) +  2t \langle \nabla f , \nabla ( f_t + af + \mathcal{R} - | \nabla f |^2 ) \rangle \\
  &   \hspace{1.5cm} - \alpha t ( f_t + af + \mathcal{R} - | \nabla f |^2 )_t  - 2 \alpha  t (div\ h)_i f_j \\
  & \hspace{1.5cm} + \alpha t H_i f_j - \alpha t \Delta \mathcal{R} - \alpha a  t \Delta f  + 2 t R_{ij} f_{ij}.
 \end{align*} 
 Notice that 
\begin{equation}\label{eq56}
\left. \begin{array}{l}
\displaystyle - \alpha t ( f_t + af + \mathcal{R} - | \nabla f |^2 )_t\\
\displaystyle  \hspace{2cm} =  t  ( \alpha  | \nabla f |^2  -  \alpha  f_t  -  \alpha  a f  -  \alpha t\mathcal{R} )_t \\
\displaystyle  \hspace{2cm}  = t \Big( | \nabla f |^2  -  \alpha  f_t -  \alpha  a f -  \alpha  \mathcal{R}  + ( \alpha - 1)  | \nabla f |^2   \Big)_t \\ 
\displaystyle \hspace{2cm} = t  \Big(\frac{F}{t}  + ( \alpha - 1)  | \nabla f |^2   \Big)_t \\ 
\displaystyle \hspace{2cm} = F_t  -  \frac{F}{t}  + t ( \alpha - 1) ( | \nabla f |^2 )_t ,
\end{array} \right.
\end{equation}
 \begin{equation}\label{eq57}
\left. \begin{array}{l}
 \displaystyle 2t \langle \nabla f , \nabla ( f_t + af +  \mathcal{R} - | \nabla f |^2 ) \rangle + t( \alpha - 1) ( | \nabla f |^2 )_t  \\ 
\displaystyle  \hspace{1.cm}  = 2t \langle \nabla f , \nabla ( f_t + af + \mathcal{R} - | \nabla f |^2 ) \rangle  + 2 t ( \alpha - 1)  \langle \nabla f, \nabla (f_t) \rangle \\ 
\displaystyle  \hspace{3cm}  - 2 t ( \alpha - 1)  h_{ij} f_i f_j  \\  
\displaystyle \hspace{1.cm} = 2t \langle \nabla f , \nabla (  \alpha f_t + af +  \mathcal{R} - | \nabla f |^2 ) \rangle - 2 t ( \alpha - 1)  h_{ij} f_i f_j  \\  
\displaystyle \hspace{1.cm} =  - 2t \langle \nabla f , \nabla \Big( \frac{F}{t}  +( \alpha - 1)  ( af +  \mathcal{R})  \Big) \rangle - 2 t ( \alpha - 1)  h_{ij} f_i f_j  \\ 
 \displaystyle \hspace{1.cm} =  - 2 \langle \nabla f , \nabla F  \rangle  - 2 t ( \alpha - 1)  \langle \nabla f , \nabla \mathcal{R}  \rangle  - 2 t ( \alpha - 1)  a | \nabla f|^2 \\
 \displaystyle  \hspace{3cm} - 2 t ( \alpha - 1)  h_{ij} f_i f_j
\end{array} \right.
\end{equation}
and
\begin{equation}\label{eqn38}
\left. \begin{array}{l}
\displaystyle - \alpha a t \Delta f  = a t  ( \alpha  | \nabla f |^2  -  \alpha  f_t  -  \alpha  a f  -  \alpha \mathcal{R} ) \\
\displaystyle  \hspace{1.4cm}  = aF   + t ( \alpha - 1) ( | \nabla f |^2 ).
\end{array} \right.
\end{equation}
With (\ref{eq56})--(\ref{eqn38}) we get 
\begin{equation}\label{eqn39}
\left. \begin{array}{l}
\displaystyle  \Delta F - F_t \\
 \displaystyle   \hspace{1cm} = t \Big( 2 f_{ij}^2  - 2 \alpha h_{ij} f_{ij} \Big)  - 2 \langle \nabla f , \nabla F  \rangle - \frac{F}{t} - 2 t ( \alpha - 1)  h_{ij} f_i f_j  \\ \ \\
\displaystyle   \hspace{1cm}  - \alpha t ( 2 (div\ h)_i f_j - H_i f_j )  - 2 t ( \alpha - 1)  \langle \nabla f , \nabla \mathcal{R}  \rangle - \alpha t \Delta \mathcal{R} \\ \ \\
\displaystyle   \hspace{1cm}  + 2 t R_{ij} f_{ij} - 2 t ( \alpha - 1)  a | \nabla f|^2 +  t ( \alpha - 1)  a | \nabla f|^2 + a F.
\end{array} \right.
\end{equation}
We now choose any two real numbers $ p, q > 0$ such that  $ \frac{1}{p} + \frac{1}{q} = \frac{1}{\alpha}$ so that we can write 
   \begin{align*}  
   2 f_{ij}^2 - 2 \alpha h_{ij} f_{ij}  & = \frac{2 \alpha }{p} f_{ij}^2 + 2   \Big(\frac{ \alpha}{q} f_{ij}^2- h_{ij}  f_{ij} \Big) \\
   & \geq \frac{2 \alpha }{p} f_{ij}^2  - \frac { \alpha q}{2}  h_{ij}^2,
   \end{align*}
   where we have used completing the square method to arrive at the last inequality. Also by Cauchy-Schwarz inequality we have $  f_{ij}^2 \geq \frac{1}{n}  ( \Delta f)^2.$
   We can also write the boundedness condition on $h_{ij}$  as $ -( k_2 + k_3) g \leq h_{ij} \leq ( k_2 + k_3) g$ so that 
   $$ \sup_M  |h_{ij}|^2 \leq n ( k_2 + k_3)^2$$
   since $h_{ij}$ is a symmetric tensor. Therefore we have 
 \begin{equation}\label{eq58}
 t \Big( 2 f_{ij}^2  - 2 \alpha h_{ij} f_{ij} \Big) \geq \frac{2 \alpha t }{ np} ( \Delta f )^2 - \frac{\alpha n q}{2}t (k_2 + k_3)^2. 
 \end{equation}
 Notice also that
 \begin{align*} 
 \alpha t \Big( 2 (div\ h)_i f_j - H_i f_j \Big)  &= 2 \alpha t \Big( div \ h - \frac{1}{2} \nabla H \Big) f_j \\
 & = 2 \alpha t \Big( g^{ij} \nabla_i h_{jl} - \frac{1}{2} g^{ij} \nabla_i h_{ij} \Big) \nabla_j f \\
  &\leq  2 \alpha t \Big( \frac{3}{2} |g| | \nabla h| \Big) | \nabla f| \\
 & \leq 3  \alpha t n^{\frac{1}{2} } k_4 | \nabla f| .
 \end{align*} 
Putting together the last inequality, (\ref{eq58}) and (\ref{eqn39}) with the assumption  that $R_{ij} \geq -k_1g$, we arrive at  
 \begin{align*}
 ( \Delta - \partial_t ) F  & \geq  - 2 \langle \nabla f , \nabla F  \rangle  -   \frac{2 \alpha t }{ np} (\Delta f )^2  - \frac{F}{t} + aF - 2 t ( \alpha - 1) k_3 | \nabla f |^2 \\
 &  - 2 t k_1 | \nabla f|^2  - 3  \alpha t n^{\frac{1}{2} } k_4 | \nabla f| - \frac{\alpha n q}{2} t (k_2 + k_3)^2  - \alpha t \Delta \mathcal{R}\\
 &  - 2 t ( \alpha - 1)  \langle \nabla f , \nabla \mathcal{R}  \rangle -  t ( \alpha - 1) a |\nabla f|^2.
 \end{align*}
 Our calculation is valid in the ball $ \mathcal{B}_{2 \rho, T}$. Hence the desired claim follows.
\end{proof}
\section{Proof of Theorem \ref{thm54}}
In order to prove Theorem \ref{thm54} we will make use of the lemma above and the assumptions that  the sectional curvature,  $ \| \nabla h \|, | \mathcal{R} |, |\nabla \mathcal{R}|  $ and $  | \Delta \mathcal{R} |$ are uniformly bounded on $M \times [0, T]$.
Then we write equation (\ref{eq55}) as 
\begin{equation}\label{eqp	1}
\left. \begin{array}{l}
\displaystyle ( \Delta  - \partial_t) F  \geq - 2 \langle \nabla f, \nabla F  \rangle -   \frac{2 \alpha t }{ np} (\Delta f )^2  - \frac{F}{t}   + aF - C_1 t | \nabla f |^2     \\
 \displaystyle  \hspace{2cm} -  C_2 t | \nabla f |    - 2k_1 t | \nabla f |^2  - \frac{\alpha n q }{2}  t ( k_2 + k_3)^2 ,
\end{array} \right.
\end{equation}
where constants $C_1 >0$ depends on $\alpha$, \ $ \max \{a, 0 \}$,  $\sup| h|$ and $ \|\nabla h\|$, and $C_2 >0$ depends on $\alpha, \ n$ and the space-time bounds of $ \|\nabla h\|, |\nabla \mathcal{R}|,   | \Delta \mathcal{R} |$. We have used the following inequality 
$$ 3 \alpha n^{\frac{1}{2}} k_4 t | \nabla f | \leq 2t k_4    | \nabla f |^2 + 2 \alpha^2 nt k_4. $$ 
Furthermore, by using 
$$ - C_2 t | \nabla f |  \geq - \delta^{-1} t C_2^2 - \delta t  | \nabla f |^2 $$
for any number $\delta > 0$, we have 
\begin{equation}\label{eqp2}
\left. \begin{array}{l}
\displaystyle ( \Delta  - \partial_t) F  \geq - 2 \langle \nabla f, \nabla F  \rangle -   \frac{2 \alpha t }{ np} (\Delta f )^2  - \frac{F}{t}   + aF - C_3 t | \nabla f |^2     \\
 \displaystyle  \hspace{2cm} -  C_4 t   - 2k_1 t | \nabla f |^2  - \frac{\alpha n q }{2}  t ( k_2 + k_3)^2 ,
\end{array} \right.
\end{equation}
where $C_3 >0$ depends on $C_1$ and $\delta$ and $C_4$ depends on $C_2$ and $\delta$.
  
   \subsection*{Estimating the cut-off function}
A natural function
 that will be defined on $M$ is the distance function from a given point.  Namely, let $ y \in M$ and define $d(x, y)$ for all $ x\in M$,  where $d(\cdot, \cdot )$ is the geodesic distance. Note that $d$ is  everywhere continuous except on the cut locus of
   $y$ and on the point where $x$ and $y$ coincide. It is then easy to see that 
$ | \nabla d | = g^{ij} \partial_i d \partial_j d = 1$ on $M  \setminus \{ \{ y \}  \cup cut(y) \} .$
Let $d(x, y, t)$ be the geodesic distance between $x$ and $y$ with respect to the metric $g(t)$, 
 we  define a smooth cut-off function $ \varphi(x, t)$  with support in the geodesic ball 
$$ \mathcal{B}_{ 2\rho, T} := \{ ( x, t) \in M \times (0, T] : d(x, y, t) \leq 2\rho \}.$$
For any $C^2$-function $ \psi( s)$ on $[0, + \infty )$ with $ \psi(s) = 1$ on $ 0 \leq s \leq 1$ and $ \psi(s) = 0$ on $ 2 \leq s \leq + \infty$
such that 
$ - C_5 \leq \psi'(s) \leq 0,  \ - C_6 \leq  \psi''(s) \leq  C_6$  and  $- C_6  \psi  \leq | \psi'|^2 \leq C_6 \psi,$
where $C_5, C_6$ are absolute constants. Let $\rho \geq 1$ and define a smooth function 
$$ \varphi(x, t) = \psi \Big( \frac{d( x, p, t)}{\rho } \Big)  \ \ \ \ and  \ \  \ \varphi \Big|_{ \mathcal{B}_{ 2\rho, T} } =1 .$$
We will apply maximum principle and invoke Calabi's trick to assume everywhere smoothness of $ \varphi(x, t)$ since $ \psi(s)$ is in general Lipschitz (see the argument of Li-Yau in \cite{[LY86]}).
We need Laplacian comparison theorem \cite{[SY94]} to do some calculation on $\varphi(x, t)$.  Let $M$ be a complete $n$-dimensional Riemannian manifold whose Ricci curvature is  bounded from below by $ Rc \geq -(n-1) k_1 $ for some constant $ k_1 \in \mathbb{R}$, then the Laplacian  of the distance function satisfies 
$$\Delta d(x, y) \leq (n-1)\sqrt{|k_1|} \coth ( \sqrt{|k_1|} \rho),  \ \ \forall x \in M\ \ d(x, y) \geq \rho.$$
We need the following calculation
\begin{align*}
 \frac{| \nabla \varphi |^2}{ \varphi} = \frac{| \psi' |^2 \cdot | \nabla d|^2 }{\rho^2 \psi} \leq \frac{C_6}{\rho^2}
  \end{align*}
and by the Laplacian comparison theorem we have
 \begin{align*}
 \Delta \varphi = \frac{ \psi' \Delta d}{\rho} + \frac{ \psi'' | \nabla d |^2}{\rho^2} & \geq -\frac{C_6}{\rho}(n-1) \sqrt{k_1} \coth( \sqrt{k_1} \rho ) - \frac{C_6}{\rho^2}\\
  \displaystyle &   \geq -\frac{C_6\sqrt{k_1}}{\rho}  - \frac{ C_6}{\rho^2}. 
  \end{align*}
 Next is to estimate time derivative of $\varphi$: consider a fixed smooth path $\gamma :[a, b] \to M$ whose length at time $t$ is given by $d(\gamma) = \int_a^b |\gamma'(t)|_{g(t)} dr$, where $r$ is the arc length. Differentiating we get 
 $$ \frac{\partial}{\partial t} (d(\gamma)) = \frac{1}{2} \int_a^b \Big|\gamma'(t) \Big|^{-1}_{g(t)} \frac{\partial g}{\partial t} \Big(\gamma'(t), \gamma'(t)\Big) dr = \int_\gamma h_{ij}(X, X) dr,$$
 where $X$ is the unit tangent vector to the path $\gamma$.  Now 
\begin{align*}
\frac{\partial}{\partial t} \varphi & = \psi ' \frac{1}{\rho} \frac{d}{dt} (d(t)) = \psi' \frac{1}{\rho} \int_\gamma h_{ij}(X, X) dr \\
& \leq \frac{ \sqrt{C_6} \psi^{\frac{1}{2}}}{\rho} (k_2 +k_3)^2 \int_\gamma dr = \sqrt{C_6} (k_2 +k_3)^2.
  \end{align*} 
Hence  we denote
$$ ( \Delta - \partial_t ) \varphi \geq  \Big( - \frac{C_6 \sqrt{k_1}}{\rho} - \frac{C_6 }{\rho^2}  - \sqrt{C_6} (k_2 +k_3)^2\Big) =: C_7. $$
which will be used in the proof of our result.

\begin{proof} ({\bf of Theorem \ref{thm54}}).
Using the same notations as in the last lemma, we write $ \widetilde{K} = ( k_2 + k_3)^2$. For a fixed $ \tau \in (0, T]$ and a smooth cut-off function $ \varphi(x,t)$ (chosen as before), we now estimate the inequality (\ref{eqp2}) at the point $( x_0, t_0) \in \mathcal{B}_{2\rho, T } \subset ( M \times [0,T])$ such that $ d( x, x_0, t) < 2 \rho$.
The argument follows;
 \begin{equation}\label{eq38}
 (\Delta - \partial_t )( \varphi F) =  2  \nabla \varphi \nabla F  + \varphi (\Delta - \partial_t )F+  F (\Delta - \partial_t ) \varphi. 
  \end{equation}
  Suppose $ ( \varphi F)$ attains its maximum value at $( x_0, t_0) \in M \times [0, T]$, for $t_0 >0$.  
  If $ ( \varphi F)(x_0, t_0) \leq 0$ for any $\rho \geq 1$, then the result holds trivially in $M \times [0, T]$ and we are done. Hence we may assume without loss of generality that there exists $ ( \varphi F)(x_0, t_0) > 0$. Then since $ ( \varphi F)(x,0) = 0$ for all $x \in M$, we have by the maximum principle that
  \begin{equation}\label{eq39}
  \nabla ( \varphi F) ( x_0, t_0) = 0, \ \ \ \frac{\partial}{\partial t}( \varphi F) ( x_0, t_0) \geq 0, \ \ \  \Delta ( \varphi F) ( x_0, t_0) \leq 0,
  \end{equation} 
  where the function $( \varphi F)$ is being considered with support on $ \mathcal{B}_{2\rho} \times [0, T]$ and we have assumed that $ ( \varphi F) ( x_0, t_0) >0$ for $ t_0 >0$. By (\ref{eq39}) we notice that
  $$ (\Delta - \partial_t )(  \varphi F)( x_0, t_0) \leq 0.$$
 Hence we have by using the inequality (\ref{eqp2}) and equation (\ref{eq38}):
 \begin{align}\label{eq3.10}
  \left. \begin{array}{l}
 \displaystyle 0 \geq ( \Delta - \partial_t ) ( \varphi F) \\
   \displaystyle \hspace{0.5cm} \geq  2  \nabla \varphi \nabla F  + C_7 F +   \  \varphi \Big \{  \frac{ 2 \alpha}{np} t_0 (\Delta f )^2  - 2 \langle \nabla f, \nabla F  \rangle  - \frac{F}{t_0} + a F\\   
  \displaystyle  \hspace{0.5cm}  - C_3 t_0 | \nabla f |^2 -  C_4t_0    - 2k_1 t_0 | \nabla f |^2 - \frac{\alpha n q }{2}  t_0 ( k_2 + k_3)^2  \Big \}.
  \end{array} \right. 
  \end{align}
The above inequality  holds in the part of  $ \mathcal{B}_{2\rho, T }$ where $ \varphi(x, t)$ is strictly positive ($ 0 < \varphi(x, t) \leq 1$ ). Notice that since $\nabla (\varphi F) = 0$, the product rule tells us that we can always replace $-F\nabla \varphi$ with $\varphi \nabla F$ at the maximum point $(x_0, t_0)$. Indeed, the following identities hold
  \begin{align*}
   2 \nabla  \varphi \nabla F =  2 \varphi \frac{\nabla \varphi}{\varphi} \nabla F = -2 F \frac{C_6}{\rho^2}  F
   \end{align*}
   \begin{align*}
   - 2 \varphi \nabla F \cdot \nabla f = 2 F \nabla \varphi \cdot \nabla f  =  2 F | \nabla f| \varphi \frac{| \nabla \varphi |}{\varphi } \geq  -2 \frac{ \sqrt{C_6}}{\rho} | \nabla f|  \varphi^{\frac{1}{2}} F
   \end{align*}
Multiplying (\ref{eq3.10})  by $( t_0 \varphi)$,  after some simple calculations involving the last two identities at the maximum point we  get 
   \begin{align*}
    0 \geq & -2 t_0 \frac{C_6}{\rho^2} \varphi  F -  \varphi^2 F - 2 t_0 \frac{ \sqrt{C_6}}{\rho} | \nabla f|  \varphi^{\frac{3}{2}} F + C_7 t_0 \varphi F  + a \varphi^2 t_0 F \\
    &+  \varphi  \frac{ 2 t_0^2}{n}   \Big( \frac{ \alpha}{p} ( \varphi | \nabla f |^2 - \varphi ( f_t + a f + \varphi \mathcal{R} ) \Big)^2   -     \frac{ \alpha  n q }{2} t^2_0 \widetilde{K} \varphi^2 \\
    &  - C_3 t^2_0  \varphi^2| \nabla f |^2 -  C_4 \varphi^2t^2_0    - 2k_1 t^2_0  \varphi^2| \nabla f |^2 \\ 
 \geq & -2 t_0 \frac{C_6}{\rho^2} \varphi  F -  \varphi^2 F - 2 t_0 \frac{ \sqrt{C_6}}{\rho} | \nabla f|  \varphi^{\frac{3}{2}} F + C_7 t_0 \varphi F  + a \varphi^2 t_0 F \\
    &+  \varphi  \frac{ 2 t_0^2}{n}   \Big( \frac{ \alpha}{p} ( \varphi | \nabla f |^2 - \varphi ( f_t + a f + \varphi \mathcal{R} ) \Big)^2   -     \frac{ \alpha  n q }{2} t^2_0 \widetilde{K} \varphi^2 \\
    &  - C_8 t^2_0  \varphi^2| \nabla f |^2,
 \end{align*}
where $C_8$ depends on $C_3,  \ C_4$ and $k_1$.
Using a similar technique as in Li-Yau paper \cite{[LY86]},  when $  t_0 > 0$, let $ y = \varphi  | \nabla f |^2 $ and $ z = \varphi ( f_t + af + \mathcal{R})$ to obtain $ \varphi^2  | \nabla f |^2 \leq \varphi y \leq y $,  $ y^{\frac{1}{2}}( y - \alpha z) = \frac{1}{t_0}  | \nabla f |  \varphi^{\frac{3}{2}} F $ and $ \varphi F = t_0 ( y - \alpha z )$. We get
  \begin{align}\label{eqp6}
  \left. \begin{array}{l}
  \displaystyle   0 \geq  \frac{ 2 t_0^2}{n}   \Bigg(  \frac{ \alpha}{p}  ( y-z)^2  -\frac{C_8}{2}n y - \frac{n \sqrt{C_6}}{\rho} y^{\frac{1}{2}}(y- \alpha z) \Bigg)   \\
    \displaystyle \hspace{0.5cm}  - \frac{ \alpha  n q }{2} t^2_0 \widetilde{K} \varphi^2+ \Big( C_7  t_0 - 2 t_0 \frac{C_6}{\rho^2} - 1 + a t_0 \Big)( \varphi F).
    \end{array} \right. 
  \end{align}
 Notice by direct calculation that 
 \begin{equation*}
\left. \begin{array}{l}
 \displaystyle  ( y- z)^2  = \Big[ \frac{1}{\alpha}( y-  \alpha z) + \frac{\alpha-1}{\alpha}y \Big]^2 \\\ \\
 \displaystyle  \hspace{1.5cm} =  \frac{1}{\alpha^2}( y-  \alpha z)^2 +  \frac{(\alpha-1)^2}{\alpha^2} y^2 +  \frac{2(\alpha-1)}{\alpha^2} y( y-  \alpha z).
\end{array} \right.
\end{equation*}
 Then, the first term in the right hand side the inequality (\ref{eqp6}) can be simplified as follows: 
 \begin{align*}
      \frac{ 2 t_0^2}{n}  & \Bigg \{  \frac{ \alpha}{p} \Bigg[ ( y-z)^2  - \frac{C_8 np}{ 2\alpha}  y   - \frac{n p}{\alpha}  \frac{ \sqrt{C_6}}{\rho}y( y - \alpha z)  \Bigg] \Bigg \}   \\
     & \hspace{.5cm} =   \frac{ 2 t_0^2}{n}   \Bigg \{  \frac{ \alpha}{p} \Bigg[  \frac{1}{\alpha^2} ( y - \alpha z)^2 +   \Bigg( \frac{( \alpha-1)^2}{\alpha^2} y^2   -\frac{C_8 np}{ 2\alpha}  y  \Bigg) \\  
      &  \hspace{.5cm} +   \Bigg( \frac{2(\alpha-1)}{\alpha^2} y - \frac{np}{\alpha}  \frac{\sqrt{C_6}}{\rho} y^{\frac{1}{2}} \Bigg)( y-  \alpha z) \Bigg] \Bigg \} \\
  & \hspace{.5cm} \geq   \frac{ 2 t_0^2}{n}   \Bigg \{   \frac{1}{\alpha p}( y-  \alpha z)^2 - \frac{C^2_8 \alpha n^2 p }{ 16( \alpha-1)^2 } 
          -  \frac{ C_6 \alpha n^2 p}{8  \rho^2 (\alpha-1)} ( y-  \alpha z ) \Bigg \}  \\
 & \hspace{.5cm} =  \frac{ 2}{\alpha np} ( \varphi F)^2   - \frac{C_8^2 \alpha n p}{8 (\alpha-1)^2 } t_0^2  - \frac{ C_6 \alpha n p}{4  \rho^2 (\alpha-1)} t_0 ( \varphi F ).\\
    \end{align*}
We have used the inequality of the form   $ ax^2 - b x \geq -\frac{b^2}{ 4a}, \  ( a, b > 0)$, to compute
\begin{align}
\displaystyle \frac{(\alpha-1)^2}{\alpha^2} y^2  - \frac{C^2_8np D}{2 \alpha} y &\geq - \frac{ C_8  n^2 p^2}{ 16 (\alpha-1)^2 } ,\\
 \displaystyle\frac{2(\alpha-1 )}{\alpha^2} y - \frac{np}{\alpha} \frac{\sqrt{C_2}}{\rho} y^{\frac{1}{2}} & \geq -  \frac{ C_6 n^2 p^2}{8(\alpha-1) \rho^2}.
\end{align}
 Therefore putting all these together into  (\ref{eqp6}), we get a quadratic polynomial in $( \varphi F)$
                   \begin{align*}
    0 \geq & \frac{ 2}{\alpha np} ( \varphi F)^2 + \Bigg( C_7  t_0 - 2 t_0 \frac{C_6}{\rho^2} - 1 + a t_0 - \frac{ C_6 \alpha n p}{4  \rho^2 (\alpha-1)} t_0 \Bigg) ( \varphi F )  \\
    &- \Bigg(  \frac{ C_8^2 \alpha n p}{ 8 (\alpha-1)^2 } t_0^2 + \frac{ \alpha  n q }{2} t^2_0 \widetilde{K} \varphi^2  \Bigg). 
         \end{align*}   
Then we develop a  formula  for quadratic inequality of the form $  ax^2 + b x + c \leq 0$,  for $x \in \mathbb{R}$. Note that  when   $ a > 0$ and $ c < 0$, then  $ b^2 - 4ac > 0$ and we have an upper  bound 
 \begin{align}\label{eqp9}
       x \leq  \frac{-b + \sqrt{b^2 - 4ac}}{ 2a}  \leq \frac{1}{a} \Big \{ - b + \sqrt{-ac} \Big \} . 
  \end{align}  
  The next is to make more explicit the term 
    \begin{align*}
    b &:=  \Big(C_7  t_0 - 2 t_0 \frac{C_6}{\rho^2} - 1 + a t_0 - \frac{ C_6 \alpha n p}{4  \rho^2 (\alpha-1)} t_0\Big)\\
     & = \Big( - \frac{C_6 \sqrt{k_1}}{\rho}t_0  - \frac{C_6 }{\rho^2}t_0   - \sqrt{C_6} \widetilde{K}t_0   - 2 t_0 \frac{C_6}{\rho^2} -1 + a t_0 - \frac{ C_6 \alpha n p}{4  \rho^2 (\alpha-1)} t_0\Big) \\
   &= - \Big(\frac{C_9}{\rho^2} t_0 \Big( \frac{ \alpha p}{\alpha - 1} + \rho \sqrt{k_1} + \rho^2 \widetilde{K} \Big) -a t_0 + 1 \Big), 
    \end{align*}
 where $C_9 >0$ depends on $C_6$ and $n$. Hence, we have by applying (\ref{eqp9})  
 \begin{align*}
 \varphi F & \leq  \frac{\alpha n  p}{2}  +  \frac{\alpha n  p}{2 } \Big \{ \frac{C_9}{\rho^2} t_0 \Big( \frac{ \alpha p}{\alpha - 1}  + \rho \sqrt{k_1}  +\rho^2( k_2 + k_3 )^2 \Big) - a t_0 \Big \} \\
 & +   \frac{ \alpha n p }{ 4(\alpha -1) } C_8 t_0  +\frac{\alpha n }{2} ( k_2 + k_3) t_0 \varphi \sqrt{pq}.
 \end{align*}
 To obtain the required bound on $F(x, \tau)$ for an appropriate range of $ x \in M$, we take $ \varphi(x, \tau) \equiv 1$ whenever $d( x , x_0, \tau ) < 2 \rho$ and since $( x_0, t_0)$ is the maximum point for $( \varphi F)$ in $ \mathcal{B}_{2\rho, T}$, we have 
 $$ F(x, \tau) = ( \varphi F)( x, \tau) \leq ( \varphi F)(x_0, t_0)$$
 for all $ x \in M$, such that $d( x , x_0, \tau ) < \rho$ and $ \tau \in ( 0, T]$ was arbitrarily chosen, then we have the conclusion in a more compact way, that 
  \begin{align}\label{eq511}
\displaystyle \sup_{ x \in \mathcal{B}_{2\rho}} \Big \{ | \nabla f |^2 - \alpha f_t  -  \alpha a f - \alpha \mathcal{R} \Big \} \leq \frac{\alpha np}{2 t} (1 - a t) + C_{10}, 
\end{align}
where $C_{10}$ depends on $ \alpha, \tau, \rho, k_1, k_2, k_3, n, p$ and $q$. This ends the proof of Theorem \ref{thm54}.
\end{proof}
\begin{Remark}\label{rmk31}
Global estimate follows by letting $ \rho \rightarrow \infty$ for all $ t > 0$. For instance, if we set $ p = 2 \alpha = q $ and allow  $ \rho$ goes to infinity, we have the estimate
 \begin{align}\label{eq512}
  \frac{| \nabla u|^2}{u^2} - \alpha \frac{ u_t}{u} -  \alpha a \log u - \alpha \mathcal{R}  \leq  \frac{\alpha^2 n  }{t}  + C_{11}
 \end{align}
where $C_{11}$ is an absolute constant depending on $n, \tau , \alpha$ and the upper bounds of $|Rc|, | \nabla \mathcal{R}|, $ $| \Delta \mathcal{R}|, |h|, | \nabla h|$ and $ - \min \{a, 0 \}$.
\end{Remark} 
As an application of the  global gradient estimates derived in Theorem \ref{thm54}, we obtain the followng result for the corresponding Harnack estimates.
\begin{corollary}\label{corN5}(Hanarck estimates).
With the same assumption as in Theorem \ref{thm54}. The following estimate
 \begin{align}\label{eq514}
\displaystyle \frac{u(x_1, t_1)}{u(x_2, t_2)^{e^{-a(t_2 -t_1)}}}  \leq  \Big(  \frac{t_2}{t_1}\Big)^{
\alpha n} \exp \Bigg \{ \int_0^1 \Bigg( \frac{ \alpha | \dot{\gamma}(s) |^2  + 4( t_2 - t_1 )^2 C_{12}}{ 4 (t_2 -  t_1)}    \Bigg) ds \Bigg \}
\end{align}
holds for  all $ (x, t) \in M \times (0, T]$, where $C_{12}$ is an absolute constant depending on $n, \tau , \alpha$ and the upper bounds of $|Rc|, |\mathcal{R}|, | \nabla \mathcal{R}|, $ $| \Delta \mathcal{R}|, |h|, | \nabla h|$ and $ - \min \{a, 0 \}$. 
The  space-time path $ \gamma :[t_1, t_2] \to M$ connects points $x_1= \gamma(t_1)$ and $x_2 = \gamma(t_2)$ in $M$. The norm $| \cdot|$ depends on $t$.
\end{corollary}

\proof ({\bf of Corollary \ref{corN5}}).
Equation (\ref{eq512}) implies
$$ f_t \geq \frac{1}{\alpha} |\nabla f|^2 - \frac{\alpha n  }{t}  - a f - \mathcal{R} - \frac{1}{\alpha}  C_{11}.$$
Straight computation yields
\begin{align*}
e^{at_2} & f(x_2, t_2) -  e^{a t_1} f(x_1, t_1) = \int_{t_1}^{t_2} \frac{d}{dt} \Big( e^{at} f( \gamma(t), t) \Big) dt\\
& =  \int_{t_1}^{t_2} \Big \{ e^{at} (f_t + \langle \nabla f( \gamma(t), t), \dot{\gamma}(t) \rangle) + a  e^{at} f \Big \}\\
& \geq  \int_{t_1}^{t_2} \Big \{ e^{at} \Big( \frac{1}{\alpha} |\nabla f|^2 - \frac{\alpha n  }{t}  - \mathcal{R} - \frac{1}{\alpha}  C_{11} + \langle \nabla f( \gamma(t), t), \dot{\gamma}(t) \rangle \Big)\Big \}\\
& \geq - e^{at} \Big( \int_{t_1}^{t_2} \frac{\alpha | \dot{\gamma}(t) |^2}{4(t_2 - t_t)} dt  + \log \Big(\frac{t_2}{t_1}\Big)^{\alpha n} + C_{12}(t_2 - t_1) \Big),
\end{align*}
where we have used inequality of the form $Ay^2 + B y \geq - B^2 / 4A.$ Positive $C_{12}$ depends on $\alpha, C_{11}$ and the uniform bound for $|\mathcal{R}|$. 
 Multiplying both sides by $e^{-a t_1}$ we have 
\begin{align*}
 f(x_1, t_1)  - e^{a(t_2-t_1)} f(x_2, t_2) = \log \Bigg(\frac{u(x_1, t_1)}{u(x_2, t_1)^{e^{a(t_2-t_1)}}} \Bigg).
\end{align*} 
By exponentiation we arrive at 
\begin{align*}
u(x_1, t_1) \leq u(x_2, t_1)^{e^{a(t_2-t_1)}} \Big(\frac{t_2}{t_1}\Big)^{\alpha n} \exp \Big \{  \int_{t_1}^{t_2} \frac{\alpha | \dot{\gamma}(t) |^2}{4(t_2 - t_t)} dt + C_{12}(t_2 - t_1) \Big \},
\end{align*}
which concludes the proof of the corollary.
\qed

\section*{Acknowledgements}

The author wishes to thank  the anonymous referees for their useful comments. 
His research is supported by TETFund of Federal Government of Nigeria and University of Sussex, United Kingdom.

\end{document}